\documentclass{amsart}
\usepackage{amsthm, amssymb, lmodern}
\usepackage{microtype}
\usepackage[a4paper, hmargin={2cm}]{geometry}
\usepackage{paralist}
\usepackage[matrix, arrow]{xy}
\usepackage[english]{babel}
\usepackage{csquotes}
\usepackage{xcolor}

\usepackage[backend=bibtex, firstinits=true, style=alphabetic, doi=false, url=false, eprint=true, sorting=nyt, maxnames=3, maxbibnames=99, isbn=false]{biblatex}
\newbibmacro{string+doiurl}[1]{%
	\iffieldundef{doi}{%
		\iffieldundef{url}{%
			\iffieldundef{eprint}{%
				#1%
			}{\href{https://arxiv.org/abs/\thefield{eprint}}{#1}}
		}{\href{\thefield{url}}{#1}}
	}{\href{https://dx.doi.org/\thefield{doi}}{#1}}
}

\renewcommand*{\bibfont}{\footnotesize}
\renewbibmacro{in:}{}
\DeclareFieldFormat{eprint:arxiv}{arXiv\addcolon\space \nolinkurl{#1}}
\DeclareFieldFormat[book]{title}{\usebibmacro{string+doiurl}{\mkbibemph{#1}}}
\DeclareFieldFormat[article,incollection, inproceedings,online, thesis,misc]{title}{\usebibmacro{string+doiurl}{#1}}

\usepackage{hyperref}
\usepackage{cleveref}
\hypersetup{colorlinks=true,citecolor=blue, urlcolor=violet, breaklinks=true, linktoc=page,bookmarksdepth=2}

\theoremstyle{definition}
\newtheorem{defn}{Definition}[section]
\newtheorem{rmk}[defn]{Remark}

\theoremstyle{plain}
\newtheorem{prop}[defn]{Proposition}
\newtheorem{lemma}[defn]{Lemma}
\newtheorem{thm}[defn]{Theorem}
\newtheorem{cor}[defn]{Corollary}

\newcommand{\lra}{\longrightarrow}
\newcommand{\lla}{\longleftarrow}
\newcommand{\hra} {\hookrightarrow}
\newcommand {\ssra} {\Rightarrow}
\newcommand{\cra}{{\ooalign{$\longrightarrow$\cr\hidewidth$\circ$\hidewidth\cr}}}
\newcommand{\cla}{{\ooalign{$\longleftarrow$\cr\hidewidth$\circ$\hidewidth\cr}}}

\newcommand{\sr}{\stackrel}

\newcommand{\W}{\mathbb{W}}
\newcommand{\F}{\mathbb{F}}
\newcommand{\Z}{\mathbb{Z}}
\newcommand{\ZZ}{\mathbb{Z}}
\newcommand{\NN}{\mathbb{N}}

\newcommand{\HH}{\mathrm{H}}

\newcommand{\HF}{\HH\F}

\newcommand{\sK}{\mathcal{K}}

\newcommand{\af}{\mathfrak{a}}
\newcommand{\fC}{\mathfrak{C}}

\newcommand{\sus}{\Sigma}
\newcommand{\tensor}{\otimes}
\newcommand{\iso}{\cong}
\newcommand{\dsum}{\oplus}
\newcommand{\we}{\simeq}

\newcommand{\Ei}{\Ealg_{\infty}}

\newcommand{\oo}{\overline}
\newcommand{\ov}{\oo{v}}
\newcommand{\h}{\hat}
\newcommand{\wt}{\widetilde}

\newcommand {\BPn} {BP\langle n \rangle}

\newcommand{\bu}{\bullet}

\newcommand{\Ab}{A_{\bu}}
\newcommand{\Bb}{B_{\bu}}

\newcommand{\Yb}{Y_{\bu}}
\newcommand{\Ub}{U_{\bu}}
\newcommand{\Mb}{M_{\bu}}
\newcommand{\Nb}{N_{\bu}}

\newcommand{\HFpb}{\HF_{p\bu}}

\newcommand{\Lmu}{\Lambda^{\mu}}
\newcommand{\Lf}{\Lambda^{f}}
\newcommand{\Lb}[1]{\Lambda_{\bu}^{#1}}
\newcommand{\Lfb}{\Lb{f}}
\newcommand{\Lmb}{\Lb{\mu}}

\newcommand{\cbu}[1]{c(#1)_{\bu}}

\newcommand{\w}[1]{w(#1)}

\newcommand{\set}[1]{\{#1\}}
\newcommand{\with}{\ \vrule\ }

\newcommand{\<}{\langle}
\renewcommand{\>}{\rangle}

\DeclareMathOperator{\Tor}{Tor}
\DeclareMathOperator*{\colim}{colim}
\DeclareMathOperator{\supp}{supp}

\newcommand{\Page}[1]{E^{#1}} 
\newcommand{\moravaE}{E} 
\newcommand{\ext}[2]{#1\<#2\>} 
\newcommand{\Ealg}{\mathbb{E}} 

\newcommand{\Defn}[1]{\textbf{#1}} 

\title[Homology of Connective Morava $\moravaE$-theory]{The Homology of Connective Morava $\moravaE$-theory with~coefficients~in~$\F_p$}

\author{Lukas Katth{\"a}n}
\address{Institut f\"ur Mathematik, Goethe-Universit\"at Frankfurt, Germany}
\email{katthaen@math.uni-frankfurt.de}
 
\author{Sean Tilson}
\address{Bergische Universit\"at Wuppertal} 
\email{seantilson@gmail.com}

\subjclass[2010]{Primary: 55P43, 55T15; Secondary: 55N20.}

\thanks{Both authors were supported by the German Research Council DFG-GRK~1916. Moreover, the first author was supported by the German Research Council, grant KA 4128/2-1.}

\begin{document}

\begin{abstract}
	Let $e_n$ be the connective cover of the Morava $\moravaE$-theory spectrum $\moravaE_n$ of height $n$.
	In this paper we compute its homology $\HH_*(e_n;\F_p)$ for any prime $p$ and $n \leq 4$ up to possible multiplicative extensions.
	In order to accomplish this we show that the K\"unneth spectral sequence based on an $\Ealg_3$-algebra $R$ is multiplicative when the $R$-modules in question are commutative $S$-algebras.
	We then apply this result by working over $BP$ which is known to be an $\Ealg_4$-algebra.
\end{abstract}

\maketitle
\tableofcontents
\section{Introduction}
In this paper we compute $\HH_*(e_n;\F_p)$ when $n\leq 4$, where $e_n$ is the connective cover of height $n$ Morava $\moravaE$-theory $\moravaE_n$ at the prime $p$. 
While $\moravaE_n$ has no homology with coefficients in $\F_p$, we find that $\HH_*(e_n;\F_p)$ contains $\HH_*(\BPn;\F_{p^n})$ as a subalgebra when $n\leq 4$.

The difference between $\HH_*(\BPn;\F_p)$ and $\HH_*(e_n;\F_p)$ is seen by the existence of classes which we denote $f_I$ for $I\subseteq\{1,\ldots n\}$.
These classes may be of interest in light of the recent work of Lawson in which he establishes that $\BPn$ can not be an $\Ei$ ring spectra when $p=2$ and $n\geq 4$, see \cite{LawsonBPEinfty}.
At the heart of his argument is the failure of $\HH_*(\BPn;\F_p)$ to be a subalgebra of the dual Steenrod algebra which is closed under the action of certain secondary Dyer-Lashof operations.
As $e_n$ is a commutative $S$-algebra, perhaps these $f_I$'s may be related to this difference in structure.

Our method of computation is relatively straightforward.
We use the K\"unneth spectral sequence based on $BP$ and show that it collapses.
We show this collapse by using a result of Baker and Richter from \cite{BakerRichterAdams} relating Massey products to Toda brackets, \Cref{thm: Baker Richter B.2}.
To apply this result we must show that the relevant K\"unneth spectral sequence is multiplicative.
This is why we base our spectral sequence on $BP$ which is known to be an $\Ealg_4$-algebra by work of Basterra and Mandell, see \cite{BasterraMandellBP}.
We then show that 
\[
\Tor_{s}^{\pi_* BP}(\HH_*(BP;\F_p),\pi_*e_n)_t \ssra \HH_{s+t}(e_n;\F_p)
\]
is multiplicative in \Cref{sec:mult}.
This uses work of Mandell from \cite{MandellHULKsmash} on categories of modules over $\Ealg_k$-algebras where $k \leq 4$.

To state our main theorem, we use the notation $[n] := \set{1,\dotsc, n}$, $\w{i} := p^i -1$ and $m(A,B) := \#\set{(a,b)\in A\times B\with a > b}$.
Further, for a ring $R$ we write $\ext{R}{x_1, x_2, \dotsc}$ for the exterior algebra with coefficients in $R$ and the indicated generators.
\begin{thm}
	When $n\leq 4$, the K\"unneth spectral sequence converging to $\HH_*(e_n;\F_p)$ collapses at the $\Page{2}$-page.
	Thus we have an isomorphism of algebras
	\[
		\Page{0}(\HH_*(e_n;\F_p))\iso
		\HH_*(BP;\F_p) \tensor_{\F_p} A/\af \tensor_{\F_p} \ext{\F_p}{\ov_{n+1},\ov_{n+2},\ldots},
	\]
	where $A = \ext{e_{n*}}{f_I\with I \subseteq [n]}$ is an exterior algebra and $\af$ is the ideal generated by the following relations:
	\begin{align*}
		u^{w(i)} u_i &\quad\text{for } 0 \leq i \leq n;\\
		u^{w(\min I)} f_I &\quad \text{for } I \subseteq [n];\\
		u_a f_{I \cup b} - u_b f_{I \cup a} &\quad \text{for } I \subseteq [n],\  a,b \in [n] \text{ with } a,b < \min(I)
	\end{align*}
	and 
	\[
		f_I\cdot f_J -
		\begin{cases}
		(-1)^{m(I \setminus i_0, J)} u_{i_0} f_{(I \setminus i_0) \cup J} &\text{ if } i_0\geq j_0 \text{ and } (I \setminus i_0) \cap J  = \emptyset; \\
		(-1)^{m(I, J \setminus j_0)} u_{j_0} f_{I \cup (J \setminus j_0)} &\text{ if } j_0\geq i_0 \text{ and } I \cap (J\setminus j_0) = \emptyset; \\
		0 & \text{ otherwise} 
		\end{cases}
	\]
	for all $I, J \subset [n]$, where $i_0 := \min(I)$, $j_0 := \min(J)$.
\end{thm}

The notation $\oo{v}_i$ is used to denote elements coming from the Koszul complex.
Recall here that $\Page{0}(\HH_*(e_n;\F_p))$ is the associated graded of $\HH_*(e_n;\F_p)$ with respect to the K\"unneth filtration.
In fact, we have a splitting of rings
\[
	\HH_*(e_n;\F_p)\iso \pi_*(\HF_p\wedge_{\BPn}e_n)\tensor_{\F_p} \HH_*(\BPn;\F_p)
\]
when $n\leq 4$, see \Cref{cor: mult splitting}.
However, there are potential multiplicative extensions involving the left hand tensor factor when $n>2$.
These issues are addressed in \Cref{subsec: extension problems}.
For example, the relation $u^{p^i-1}f_{i,j}=0$ always holds in homotopy.

We are also able to compute the K\"unneth spectral sequence converging to the homotopy of the relative smash product $\HF_p\wedge_{BP} e_n$ for $n\leq 5$.
\begin{thm}
	When $n\leq 5$, the K\"unneth spectral sequence converging to $\pi_*(\HF_p\wedge_{BP} e_n)$ collapses at the $\Page{2}$-page.
	Thus we have an isomorphism of of algebras
	\[
		\Page{0}(\pi_*(\HF_p\wedge_{BP} e_n))\iso
		A/\af \tensor_{\F_p} \ext{\F_p}{\ov_{n+1},\ov_{n+2},\ldots},
	\]
	where $A$ and $\af$ are as above.
\end{thm}

We have the following result for $n=2$:
\begin{cor}
	For $n=2$ we have that 
	\[
		\pi_*(\HF_p \wedge_{BP\<2\>}e_2)\iso\frac{\F_{p^2}[[u_1]][u,f_{1,2}]}{(u_1u^{p-1},u^{p^2-1},uf_{1,2},f_{1,2}^2)}.
	\]
\end{cor}

This corollary does not follow directly from the above theorem, but from the techniques employed in the computation.
As $BP\<2\>$ is known to be a commutative $S$-algebra we can proceed directly to the spectral sequence based on $BP\<2\>$ and avoid $BP$ entirely.
The algebraic computation of the $\Page{2}$-page, as well as the relevant Massey products, follow the same lines as the general computation in \Cref{sec: E_2 page}.
There are no possible extensions when $n=2$ for degree reasons.
The homology of $e_2$ is then recovered by tensoring with the homology of $BP\<2\>$.

The product structure that we end up with is an interesting one.
In characteristic $p$, divided power algebras frequently arise, and the structure we have here is similar.
The element $u$, which is a unit in $\moravaE_n$ but not in $e_n$, persists to give a class in $\HH_*(e_n;\F_p)$.
Many of the products that we have can be thought of as divided products with respect to $u$.

\subsection{Outline}
In \Cref{sec:mult} we show that the KSS based on an $\Ealg_3$-algebra $R$ in $S$-modules is multiplicative when the $R$-modules in question are commutative $S$-algebras.
We first recall the $\Lf$ construction of Mandell along with some of its properties.
We also recall the relevant work of the second author from \cite{KSSpaper} for establishing that a spectral sequence is multiplicative.
We then proceed to show that the spectral sequence is multiplicative by constructing a zig-zag of filtrations
\[
\xymatrix{
\Lfb(\HFpb,\cbu{e_n},\HFpb,\cbu{e_n}) \ar@{~>}[rr]^-{m_{\bu}}&
&
\Lmb(\HFpb,\cbu{e_n}),
}
\]
where $\HFpb$ is the filtration associated to the Koszul complex, $\cbu{e_n}$ is the constant filtration, and $f=([0,\frac{1}{4}]^3,[\frac{1}{4},\frac{1}{2}]^3,[\frac{1}{2},\frac{3}{4}]^3,[\frac{3}{4},1]^3)$ as an element of $\fC_3(4)$, the $4$th space in the little $3$-cubes operad.
This is sufficient to show that the differentials in the KSS satisfy the Leibniz formula and apply \Cref{thm: Baker Richter B.2} of Baker and Richter.

In \Cref{sec: E_2 page} we compute the $\Page{2}$-page of the KSS.
We compute this $\Tor$-group with it's product structure and various Massey products.
After resolving a couple of extension problems, we are able to identify these Massey products with Toda brackets and derive the collapse of the spectral sequence.
After we have this collapse result we can deduce that the target of this KSS splits as the tensor product
\[
\HH_*(e_n;\F_p)\iso \HH_*(\BPn;\F_p)\tensor_{\F_p} \pi_*(\HF_p \wedge_{\BPn} e_n).
\]

\subsection{Conventions}
We work with the model of spectra developed in \cite{EKMM}.
These are refered to as $S$-modules and $S$-algebras.
In \Cref{sec:mult} we will use $R$ to denote an $\Ealg_n$-algebra in $S$-modules whose underlying $S$-module is cofibrant.
We will use $A$ and $B$ to denote commutative $S$-algebras that receive a map of $\Ealg_1$-algebras from $R$.
The underlying $R$-module of $A$ and $B$ will also be assumed to be cofibrant.
Following Mandell, we will take $R$-modules to mean modules over a strictly associative $S$-algebra, denoted $UR$, which is homotopy equivalent to $R$ in the category of $S$-modules.
Similarly, all model categorical notions will take place in $UR$-modules instead of $R$-modules.
We also use homotopy cofibrant $R$-module to mean a module $M$ over $UR$ which is homotopy equivalent to a cell or cofibrant $UR$-module in the sense of \cite{EKMM}.

Our computation is concerned with various spectra that arise in chromatic homotopy theory.
Specifically, we work with the Brown-Peterson spectrum $BP$, the truncated Brown-Peterson spectrum $\BPn$, and connective Morava $\moravaE$-theory $e_n$.
For more information about the spectra $BP$ and $\BPn$ we recommend the modern classic \cite{GreenBook}.
We also recommend \cite{RezkHMthm} for information on Morava $\moravaE$-theory.
It is shown in \cite{GHmoduli} that $\moravaE$-theory does in fact possess the structure of an $\Ei$-algebra, and hence a commutative $S$-algebra structure.
This induces a commutative $S$-algebra structure on $e_n$.
For a survey of chromativ homotopy theory from a modern perspective see \cite{BarthelBeaudry}.

Lastly, recall that the sign convention for the Leibniz formula in chain complexes of graded modules is 
\[
\partial(ab)=\partial(a)b+(-1)^{\vert a\vert} a\partial(b)
\]
where $\vert a \vert$ is the total degree of the homogeneous element $a$.
This is the sign convention for Leibniz formulas for all differentials in spectral sequences.

\subsection*{Acknowledgments}
This project began while both authors were funded by DFG-GRK~1916 at Universit\"at Osnabr\"uck.
We are grateful to the collegial environment that was fostered there.
We would also like to thank Andrew Baker, Tobias Barthel, Tyler Lawson, and Eric Peterson for helpful conversations during this project.

\section{The multiplicativity of the K\"unneth spectral sequence for \texorpdfstring{$\Ealg_3$}{E\_3}-algebras}\label{sec:mult}
In \cite{KSSpaper}, it is shown that the K\"unneth spectral sequence
\[
\Tor_s^{\pi_*R}(\pi_*A,\pi_*B)_t \ssra \pi_{s+t}(A\wedge_R B)
\]
is multiplicative when $R$ is a commutative $S$-algebra and $A$ and $B$ are $R$-algebras.
The symmetry
\[
\tau: M\wedge_R N \lra N\wedge_R M
\] 
of the monoidal structure on $R$-modules is explicitly used.
We do not have a symmetric monoidal structure on the category of $BP$-modules itself but on the homotopy category of $BP$-modules.
This symmetric monoidal structure on the homotopy category is induced by an interchange operation which exists on the category of $BP$-modules.
First we briefly recall the work of Mandell where he constructs a point set model for the relative smash products and various interchange operations.
We then recall some definitions and results from \cite{KSSpaper} before discussing multiplicative filtrations.
Next, we show that our spectral sequence is multiplicative.
Finally, we recall \Cref{thm: Baker Richter B.2} from \cite{BakerRichterAdams} which we will use to show that the KSS collapses in \Cref{thm: collapse}.

\subsection{Monoidal structure on the category of modules over a \texorpdfstring{$\fC_n$}{C\_n}-algebra}
\label{subsec:Mandell recollection}
In this subsection, let $R$ be an $\fC_n$-algebra in the category of $S$-modules, where $\fC_n$ is the little $n$-cubes operad.
While it is usual to assume that $R$ is an $\Ealg_n$-algebra this is only a homotopically meaningful notion.
Instead, we fix an actual $\fC_n$-algebra at the outset.
Let  $\fC_{n-1}(m)$ denote the $m^{th}$-space of the little $(n-1)$-cubes operad.
In \cite{MandellHULKsmash}, Mandell constructs point-set level models for monoidal products
\[
\Lf: (R\mathrm{-mod)}^{m} \lra R\mathrm{-mod},
\]
depending on a map $f: X \lra \fC_{n-1}(m)$, where $X$ is any $CW$-complex.
Recall that Mandell works with $UR$-modules instead of $R$-modules.
Where $R$ is only assumed to be an $\fC_1$-algebra in $S$-modules, $UR$ is a homotopy equivalent associative $S$-algebra.
In particular, this allows us to work with an honest category of modules as opposed to the more general operadic definition.
See Section 2 of \cite{MandellHULKsmash} for more details on this.
Following Mandell, by $R$-modules we will implicitly mean $UR$-modules.
Similarly, all notions of cofibrancy are in terms of the model structure on $UR$-modules.

Explicitly, given a map $f: X \lra \fC_{n-1}(m)$
Mandell constructs a $UR$-$U(R^m)$-bimodule structure on $UR\wedge X_+$.
This is then used to define 
\[
\Lf(M_1,M_2,\ldots,M_m):=(UR\wedge X_+)\wedge_{U(R^m)}(M_1\wedge M_2 \wedge \cdots \wedge M_m),
\]
where $X_+$ denotes the suspension spectrum of $X$ with a disjoint basepoint.
$\Lf$ has very nice homotopical properties.
It is natural in $f$ and its homotopy type only depends on the homotopy class of $f$.
In addition, it is homotopy invariant in the $R$-module coordinates when applied to modules which are homotopy cofibrant.
In particular, this implies that $\Lambda^f$ is functorial in zig-zags when the arguments are homotopy cofibrant.
This means that $\Lf$ applied to a zig-zag of homotopy cofibrant $R$-modules also produces a zig-zag.

We are primarily interested in the case where the base space $X$ is a point and where $f$ is the constant map $\mu_m: X \to \fC_{n-1}(m)$ whose image is 
\[ \left([0,1/m]^{n-1},[1/m, 2/m]^{n-1}, \dotsc, [(m-1) / m, 1]^{n-1}\right) \in \fC_{n-1}(m),
\]
i.e., the configuration of $m$ small cubes along the diagonal of a large cube.
Whenever $m$ is clear from the context we just write $\mu$ for $\mu_m$.
The only other cases of interest, for us, will be when $X$ is an interval and $f$ is a path in $\fC_{n-1}(m)$ connecting different configurations.

Mandell uses this construction to determine what extra structure the derived category of modules over an $\Ealg_n$-algebra has when $n\in \{2,3,4\}$.
For example, he shows that
$M \widetilde{\wedge}_R N := \Lmu(M,N)$
yields a monoidal structure on the homotopy category of $R$-modules if $R$ is an $\Ealg_2$-algebra.
When $R$ is an $\Ealg_3$ this monoidal structure has a braiding.
This braiding is a symmetry when $R$ is an $\Ealg_4$.

As we are working with actual module spectra (as oposed to working in the the homotopy category), we are going rephrase his main theorem in the following way.
As a notation, we write 
\[ \xymatrix{M \ar@{~>}[r]& N} \]
to indicate that there is a zig-zag of maps connecting the $R$-modules $M$ and $N$, where all maps going on the wrong direction are weak equivalences. 
This induces a map from $\pi_*M$ to $\pi_*N$.
Similarly, we write
\[ \xymatrix{M \ar@{<~>}[r]& N} \]
when there is a zig-zag of weak equivalences connecting $M$ and $N$.
The proof of the main theorem of \cite{MandellHULKsmash} implies the following:
\begin{thm}\label{thm:mainzigzag}
	Let $R$ be an $\fC_2$-algebra and $M_1, M_2$ and $M_3$ be homotopy cofibrant $R$-modules.
	\begin{enumerate}
		\item There are zig-zags
		\[ 
		\xymatrix{
		\Lmu(\Lmu(M_1, M_2), M_3) \ar@{<~>}[r]^-{\alpha}&
		\Lmu(M_1, M_2, M_3) \ar@{<~>}[r]^-{\alpha}&
		\Lmu(M_1, \Lmu(M_2, M_3)) 
		}.
		\] 
		\item If $R$ is an $\fC_3$-algebra, there is a zig-zag
		\[ 
		\xymatrix{\Lmu(M_1, M_2) \ar@{<~>}[r]^-{\tau}& \Lmu(M_2, M_1)}.
		\]
	\end{enumerate}
\end{thm}

\begin{rmk}\label{rem: after Mandell theorem}
	These zig-zags are described explicitly in Sections 4.5 and 4.7 of \cite{MandellHULKsmash}, respectively.
	The result then follows from the combination of his Theorem 1.5, Theorem 1.7, and various properties of the $\Lf$ construction.
	In particular, this construction is natural and homotopy invariant in the $f$ coordinate as well as the module coordinates, when applied to $R$-modules that are homotopy cofibrant.
	
	More explicitly, the coherences necessary to show \Cref{thm:mainzigzag} induces a monoidal structure can be constructed as follows.
	When $n\geq 2$, there exists a path $\alpha$ from $\mu_m$ to the composition, using the operadic structure, of $\mu_r$ with $\mu_{k_1},\ldots, \ \mu_{k_r}$ such that $\sum_{i=1}^rk_i=m$.
	The existence of such an $\alpha$ then gives a natural zig-zag of functors
		\[ 
		\xymatrix{
		\Lambda^{\mu_m}\ar[r]&
		\Lambda^{\alpha}&
		\Lambda^{\mu_r\circ(\mu_{k_1},\mu_{k_2},\ldots,\mu_{k_r})} \ar[r] \ar[l]&
		\Lambda^{\mu_r} (\Lambda^{\mu_{k_1}},\Lambda^{\mu_{k_2}},\ldots, \Lambda^{\mu_{k_r}})
		}
		\]
	from $\Lambda^{\mu_m}$ to $\Lambda^{\mu_r}\circ (\Lambda^{\mu_{k_1}},\Lambda^{\mu_{k_2}},\ldots, \Lambda^{\mu_{k_m}})$.
	These maps are all weak equivalences when the arguments are homotopy cofibrant modules.
	This is the general approach to establishing coherences in the derived category.
	Such arguments also work to establish braidings and symmetries of the monoidal structure.
\end{rmk}

We will now show how to construct algebras in such a setting.
The multiplicativity of our spectral sequence will come from lifting this structure to the filtered setting.
This process is discussed in \Cref{subsec: mult filtrations over BP}.

In ordinary ring theory, a map of associative algebras $R \lra A$ only induces an $R$-algebra structure on $A$ if the image of $R$ is contained in the center of $A$.
However, if $A$ is commutative then every element of $A$ is central.
A similar argument works here so that both $\HF_p$ and $e_n$ are algebras over $BP$.
\Cref{prop: alg structure over BP} is a version of the above using $\Lmu$.

\begin{prop}\label{prop: alg structure over BP}
	Given a commutative $S$-algebra $A$, an $\Ealg_2$-algebra $R$, and a map $R \lra A$ of $\Ealg_1$-algebras in $S$-modules, there exists a product map $\Lmu(A,A) \lra A$ in the category of $R$-modules.
	This product is associative in the sense that the diagram
	\begin{equation}\label{eq:assoc}
		\xymatrix{
		\Lmu(\Lmu(A,A),A) \ar[dd] \ar@{<~>}[rr]^{\alpha}&
		& \Lmu(A,\Lmu(A,A)) \ar[dd]\\
		&
		&
		\\
		\Lmu(A,A)\ar[r]& A& \Lmu(A,A) \ar[l]
		}
	\end{equation}
	commutes upon passing to the homotopy category.
\end{prop}
In the above diagram $\alpha$ is the zig-zag of \Cref{thm:mainzigzag}.
It can be witnessed by an explicit path $\alpha$ in $\fC_1(3)$.

\begin{proof}
	The construction $\Lmu$ is defined by the coequalizer diagram
	\[
	UR \mu \wedge U(R^2) \wedge A \wedge A \rightrightarrows UR\mu \wedge A \wedge A.
	\]
	$UR\mu$ is the $UR-U(R^2)$-bimodule $UR \wedge *_+$ where the right $U(R^2)$-module structure is induced by the inclusion $\mu=([0,\frac{1}{2}]^{n-1},[\frac{1}{2},1]^{n-1}) \in \fC_{n-1}(2)$.
	We then want to construct a map of coequalizer diagrams
	\[
	\xymatrix{
	UR\mu \wedge U(R^2) \wedge A \wedge A \ar@<-.5ex>[rr] \ar@<.5ex>[rr] \ar[dd]&
	&
	UR\mu \wedge A \wedge A \ar[dd]\\
	&
	&
	\\
	UR\mu\wedge UR \wedge A \ar@<-.5ex>[rr] \ar@<.5ex>[rr]&
	&
	UR\mu \wedge A.
	}
	\]
	We first construct the square with the top arrows in the equalizer diagrams and then the square with the bottom arrows of the diagrams.
	
	The right $U(R^2)$-module structure on $UR\mu$ is constructed using the multiplication map
	\[
	U(R^2) \lra UR
	\]
	along with $\mu$.
	Thus we have a commutative diagram
	\[
	\xymatrix{
	UR\mu \wedge U(R^2) \ar[rr] \ar[dd]&
	&
	UR\mu \ar[dd]\\
	&
	&
	\\
	UR\mu\wedge UR \ar[rr]&
	&
	UR\mu.
	}
	\]
	This constructs the square with the top arrows of the coequalizer diagrams by then smashing with the product of $A$ as a commutative $S$-algebra.
	
	The left $U(R^2)$-module structure on $A\wedge A$ is the composite of
	\[
	f:U(R^2)\wedge A \wedge A \lra UR \wedge UR \wedge A \wedge A \lra UR \wedge A \wedge UR \wedge A \lra A \wedge A \wedge A \wedge A
	\]
	and the multiplication of $A$ as a commutative $S$-algebra smashed with itself.
	The map $f$ is a composition of
	\[
	U(R\wedge R) \lra UR\wedge UR,
	\]
	which is referred to as 1.2 in \cite{MandellHULKsmash}, the natural equivalence
	\[
	UR \lra R,
	\]
	the map from $R$ to $A$, and the symmetry in the underlying category of $S$-modules.
	Therefore the diagram
	\[
	\xymatrix{
	U(R^2)\wedge A \wedge A \ar[rr] \ar[dd]&
	&
	A \wedge A \ar[dd]\\
	&
	&
	\\
	UR \wedge A\ar[rr]&
	&
	A
	}
	\]
	commutes since $A$ is a commutative and associative $S$-algebra.
	We now have a map of coequalizer diagrams which naturally induces
	\[
	\Lmu(A,A) \lra A
	\]
	by definition.
	
	To see that the associativity condition is satisfied we can make a similar argument.
	As $A$ is an associative $S$-algebra we can construct a product map
	\[
	\Lambda^{\alpha}(A,A,A)=(UR\wedge I_+)\wedge_{U(R^3)}(A\wedge A \wedge A) \lra A
	\]
	by examining coequalizer diagrams as we did above.
	The invariance and functoriality of $\Lambda^{(-)}$ induces the zig-zag of equivalences
	\[
	\Lmu(\Lmu(A,A),A) \sr{\we}\lra \Lambda^{\alpha}(A,A,A) \sr{\we}\lla \Lmu(A,\Lmu(A,A)).
	\]
	Functoriality of $\Lmu(-,-)$ also gives us the maps
	\[
	\Lmu(\Lmu(A,A),A) \lra \Lmu(A,A)
	\]
	and
	\[
	\Lmu(A,\Lmu(A,A)) \lra \Lmu(A,A).
	\]
	The desired diagram commutes due to the discussion in the beginning of Section 4.5 of \cite{MandellHULKsmash} and the associativity of $A$.
\end{proof}

\subsection{Filtrations and a comparison theorem}
\label{subsec:kss recollection}
The material in this section is a brief recollection of necessary material from \cite{KSSpaper} and \cite{TilsonC2AdamsSS}.
Here we give definitions that are adapted to our situation from the standard notions.
These definitions do recover the standard constructions when $\Lf$ is replaced by $-\wedge_R-$ in the situation that $R$ is a commutative $S$-algebra.

\begin{defn}
	A \Defn{filtered spectrum} or \Defn{filtration} is a sequence of cofibrations
	\[
	\cdots \hra A_{i-1} \hra A_i \hra A_{i+1} \hra \cdots.
	\]
	We denote the single filtered object as $\Ab$.
	The \Defn{associated graded complex} of a filtered spectrum $\Ab$ is the complex of spectra
	\[
	\cdots \cla A_{i-1}/A_{i-2} \cla A_i/A_{i-1} \cla A_{i+1}/A_i \cla \cdots
	\]
	denoted by $E^0(\Ab)$.
	
\end{defn}
Our notation $A \cra B$ is an abbreviation for $A \lra \sus B$.
Here we work with increasing filtrations exclusively.
The K\"unneth filtration, which gives rise to the KSS, is such an increasing filtration.
\begin{defn}
	The \Defn{smash product of two filtrations} $\Ab$ and $\Bb$ is denoted by $\Lmb(\Ab,\Bb)$.
	The $n$th term in the filtration is 
	\[
	\Lmu_n(\Ab,\Bb):=\colim_{i+j\leq n} \Lmu(A_i,B_j).
	\]
	More generally, the iterated smash product of $r$ filtrations $\Ab^1,\Ab^2,\ldots,\Ab^r$ 
	is denoted by $\Lmb(\Ab^1,\Ab^2,\ldots,\Ab^r)$ and defined as
	\[
	\Lmu_n(\Ab^1,\Ab^2,\ldots,\Ab^r):=\colim_{i_1+\dotsb +i_r\leq n} \Lmu(A^1_{i_1}, A^2_{i_2}, \ldots,  A^r_{i_r})
	\]
\end{defn}

When $R$ is a commutative $S$-algebra, this definition is equivalent to what is given in \cite{KSSpaper}.

\begin{lemma}
	The smash product $\Lmb(\Ab^1,\Ab^2,\ldots,\Ab^r)$ of filtrations $\Ab^1,\Ab^2,\ldots,\Ab^r$ is a filtration as well. 
\end{lemma}
\begin{proof}
	Let $\Bb := \Lmb(\Ab^1,\Ab^2,\ldots,\Ab^r)$.
	We only need to show that the maps $B_i \to B_{i+1}$ are cofibrations.
	This follows from Proposition 3.10 of \cite{EKMM} and  the repeated application of the pushout product axiom.
\end{proof}

We call a filtration \Defn{projective} if the associated graded complex consists of retracts of free $R$-modules.
A filtration being \Defn{exact} is essentially equivalent to the associated graded complex being exact after applying $\pi_*$.
A filtration is \Defn{exhaustive} when $\colim \Ab\we A$.
Every exact filtration is also exhaustive by Lemma 2.2 of \cite{KSSpaper}.
For more details see Section 2.1 of \cite{KSSpaper}.
This exactness and projectivity are necessary in order to apply the following result.

\begin{thm}[{\cite[Thm 1]{KSSpaper}}]
\label{thm: Comparison Theorem}
	Suppose that we have a map $f: Y \to A$ of $R$-modules, an exact filtration $A_{\bullet} \subset A$, and a projective and exhaustive filtration $Y_{\bullet} \subset Y$.
	Assume further that $A_i = *$ and $Y_i = *$ for $i \leq -1$.
	Then there is a map of filtrations 
	\[
		\Yb \sr{f_{\bu}} \lra \Ab
	\]
	such that $\colim f_{\bu}\we f$ under the equivalences $\colim \Yb \we Y$ and $\colim \Ab \we A$.
	Furthermore, the lift $f_{\bu}$ of $f$ is unique up to homotopy of filtered modules.
\end{thm}

This result will be applied to show that the K\"unneth spectral sequence is multiplicative.
Explicit details of the construction of the K\"unneth spectral sequence can be found in Chapter 4.5 \cite{EKMM} and Sections 2 and 4 of \cite{KSSpaper}.
We will apply \Cref{thm: Comparison Theorem} in the setting of $UR$-modules with $\Lmu$.
This suffices as exactness and projectivity are homotopical notions which are preserved by $\Lmu$.

\subsection{Multiplicative filtrations and the K\"unneth spectral sequence} \label{subsec: mult filtrations over BP}

We now give a definition of multiplicative filtration.
Such filtrations will give rise to multiplicative spectral sequences in the same way that pairings of filtrations give rise to pairings of spectral sequences.
As we are not working with a commutative $S$-algebra, we will only be able to obtain a zig-zag of filtrations as opposed to a map of filtrations.
This will be sufficient though as we will be applying homotopy invariant functors.
The same is true of the coherences, such as ``associativity''.

\begin{defn}
	\label{defn: mult filt}
	We say that a filtration $\Ab$ is \Defn{multiplicative} if there is a zig-zag of maps of filtrations
	\[
	\xymatrix{
	\Lmb(\Ab,\Ab) \ar@{~>}[rr]^-{m_{\bu}}&
	&
	\Ab,
	}
	\]
	where maps of filtrations going in the wrong direction are level-wise weak equivalences.
	Further, this structure should be associative in the sense that the following diagram
	\[
		\xymatrix{
		\Lmb(\Lmb(\Ab,\Ab),\Ab) \ar@{~>}[dd] \ar@{<~>}[rr]^{\alpha}&&
		\Lmb(\Ab,\Lmb(\Ab,\Ab)) \ar@{~>}[dd]\\
		&&\\
		\Lmb(\Ab,\Ab)\ar@{~>}[r]& \Ab& \Lmb(\Ab,\Ab) \ar@{~>}[l]
		}
	\]
	commutes after passing to the homotopy category.
\end{defn}
Here, as in \eqref{eq:assoc}, $\alpha$ denotes a zig-zag of \Cref{thm:mainzigzag}.

As we will be applying a general lifting result to algebras in $BP$-modules, we can not expect for a better form of associativity in light of \Cref{prop: alg structure over BP}.
However, we will establish that the KSS applied to commutative $S$-algebras over $BP$ is multiplicative.
A useful consequence of a filtration being multiplicative is given in the following proposition:

\begin{prop}\label{prop: still mult ss}
	If $\Ab$ is a multiplicative filtration of $R$-modules, then each page of the corresponding spectral sequence is a DGA (not necessarily commutative).
	In particular, the differentials satisfy the Leibniz rule with respect to this product.
\end{prop}
\begin{proof}[Proof sketch]
	That a multiplicative filtration gives rise to a product on $\Page{r}(\Ab)$ for each $r$ follows the usual line of argument.
	That the product on the filtration is a zig-zag will not be an issue as the algebraic computation is obtained from applying homotopy invariant functors.
	The associativity of this product structure then also follows by the same argument.
	The difficult part of the above is showing that the differentials satisfy the Leibniz rule.
	The standard argument for showing this starts by representing $y,d_r(y)\in \Page{r}(\Ab)_{s,t}$ as a map from the filtered spectrum
	\[
	\Ub(r,s,n):=\cdots \to \ast \to S^{n-1} \to \cdots \sr{1}\to S^{n-1} \to D^n\sr{1}\to D^n \to \cdots
	\]
	to $\Ab$.
	The total degree of $y$ is $s+t=n$ and it is represented as a map from 
	\[
	D^n_R/S^{n-1}_R \lra A_s/A_{s-1}
	\]
	and $d_r(y)$ is represented as a map from
	\[
	S^{n-1}_R/\ast \lra A_{s-r}/A_{s-r-1}.
	\]
	We then represent the pair $(y,d_r(y))$ as a map of filtered spectra
	\[
	\wt{y}:\Ub(r,s,n) \lra \Ab.
	\]
	The spectral sequence associated to $\Ub(r,s,n)$ has two generators and one non-trivial differential between them, the obvious $d_r$.
	Given two such maps we can smash them together to form
	\[
	\wt{y}\wedge \wt{y}': \Lmb(\Ub,\Ub') \lra \Lmb(\Ab,\Ab).
	\]
	As having a map of filtrations induces a map of spectral sequences, any differential in the spectral sequence associated to $\Lmb(\Ub,\Ub')$ induces a differential in the spectral sequence associated to $\Lmb(\Ab,\Ab)$.
	The spectral sequence associated to $\Lmb(\Ub,\Ub')$ is easily computed and seen to have  $4$ generators on the $\Page{1}$-page.
	They are $y\tensor y', y\tensor d_r(y'), d_r(y)\tensor y',$ and $d_r(y)\tensor d_r(y')$.
	In this spectral sequence we have The differential $d_r(y\tensor y')=d_r(y)\tensor y'+(-1)^{s+t}y\tensor d_r(y')$ coming from the attaching map of the top cell of $S^n\wedge S^{n'}$.
	This would naturally induce differentials in the spectral sequence for $\Ab$ using a natural map 
	\[
	\Lmb(\Ab,\Ab) \lra \Ab.
	\]
	However, what we have is a zig-zag.
	This is enough though as the whole computation takes place after applying $\pi_*(-)$ and we obtain differentials in the spectral sequence associated to $\Ab$ as the composite
	\[
	\pi_*(\Ub'')\sr{f}\lra \pi_*(\Lmb(\Ub,\Ub')) \lra \pi_*(\Lmb(\Ab,\Ab))\lra \pi_*(\Ab).
	\]
	Here the map $f$ represents the differential on $y\tensor y'$ in the spectral sequence associated to $\Lmb(\Ub,\Ub')$.
	Thus we have a map of filtered abelian groups 
	\[
	\pi_*(\Ub'') \lra \pi(\Ab)
	\]
	which implies that the differentials in the spectral sequence associated to $\Ab$ satisfy the Leibniz formula.
\end{proof}
See Section 4 of \cite{TilsonC2AdamsSS} for a more detailed analysis of this.

The following lemma will be applied to the structure maps produced by \Cref{prop: alg structure over BP}.
\begin{lemma}
	\label{lemma: lift of product}
	Given a map of $R$-modules
	\[
	\Lmu(M,M) \lra M
	\]
	which is associative in the sense of \Cref{eq:assoc} there is a lift of this product structure to a map of filtrations
	\[
	\Lmb(\Mb,\Mb) \lra \Mb
	\]
	making $\Mb$ multiplicative where $\Mb$ is a projective and exact filtration of $M$ in $R$-modules.
\end{lemma}
\begin{proof}
	Given the map $\Lmu(M, M) \to M$ we wish to lift it to the filtration $\Mb$ of $M$ in $R$-modules.
	We do this using \Cref{thm: Comparison Theorem} applied to the filtration $\Lmb(\Mb,\Mb)$ and the map $\Lmu(M,M) \to M$.
	
	Since $\Mb$ is a projective filtration we also have that $\Lmb(\Mb,\Mb)$ is level-wise projective.
	This follows as when the filtrations $\Mb$ and $\Nb$ are projective we have that
	\[
	E^0(\Lmb(\Mb,\Nb))_n\we (E^0(\Mb)\tensor E^0(\Nb))_n:=\bigvee_{i+j=n}\Lmu(M_i/M_{i-1}, N_j/N_{j-1}) 
	\]
	where the tensor is the graded tensor product of complexes of spectra induced by $\Lmu$.
	Therefore we can apply \Cref{thm: Comparison Theorem} to the map $\Lmu(M,M) \to M$.
	
	When $f$ is a map from a contractible space, the construction $\Lfb$ applied to $\Mb$ in each coordinate is still a projective filtration.
	Thus we apply \Cref{thm: Comparison Theorem} to each map in the associativity diagram for $M$.
	The uniqueness, up to filtered homotopy, of our lift then ensures that the diagram commutes in the homotopy category.
	This shows that this lift is suitably associative.
\end{proof}
Such filtrations are obviously multiplicative in our sense.

\begin{prop}
\label{prop: lambda filt mult}
	Let $R$ be an $\Ealg_3$-algebra and let $\Mb$ and $\Nb$ be multiplicative filtrations of $R$-modules.
	Then the filtration $\Lmb(\Mb,\Nb)$ is a multiplicative filtration of $R$-modules.
\end{prop}
We will apply this result when $\Mb$ is a filtration of $\HF_p$ or $\HF_p\wedge BP$ by $BP$-modules and $\Nb$ is a filtration of $e_n$ by $BP$-modules.
These are the two cases where $\Lmb(\Mb,\Nb)$ have their associated spectral sequence converging to $\pi_*(\HF_p\wedge_{BP} e_n)$ or $\HH_*(e_n;\F_p)$, respectively.
In these situations, $\Mb$ is the projective and exact filtration of $\HF_p$ or $\HF_p \wedge BP$ that comes from the Koszul complex and $\Nb$ is the constant filtration of $e_n$.

\begin{proof}
	The multiplicative structure on the filtration $\Lmb(\Mb,\Nb)$ is given by the composition of zig-zags in the following diagram:
	\[
	\xymatrix{
		\Lmb(\Lmb(\Mb,\Nb),\Lmb(\Mb,\Nb)) \ar@{~>}[d]&
		&
		&
		&
		\Lmb(\Mb,\Nb)\\
		\Lmb(\Mb,\Nb,\Mb,\Nb) \ar@{~>}[rr]&
		&
		\Lmb(\Mb,\Mb,\Nb,\Nb) \ar@{~>}[rr]&
		&
		\Lmb(\Lmb(\Mb,\Mb),\Lmb(\Nb,\Nb)) \ar@{~>}[u]
	}
	\]
	Here, we repeatedly apply \Cref{thm:mainzigzag} as well as \Cref{rem: after Mandell theorem}.
	The last map is induced by the structure maps 
	\[
		\xymatrix{\Lmb(\Mb,\Mb)\ar@{~>}[r] & \Mb} \text{ and } \xymatrix{\Lmb(\Nb,\Nb)\ar@{~>}[r] & \Nb}.
	\]	
	
	That this satisfies the required associativity follows from lifting the arguments of \Cref{prop: alg structure over BP} to the filtered setting.
\end{proof}

\begin{thm}
	\label{thm: mult of KSS of E_3 alg}
	Let $R$ be an $\Ealg_3$-algebra and $A$ and $B$ be commutative $S$-algebras over $R$ which are cofibrant as $R$-modules.
	Then the K\"unneth spectral sequence
	\[
	\Tor^{R_*}_s(A_*,B_*)_t \ssra \pi_{s+t}(\Lmu(A,B))
	\]
	is multiplicative.
\end{thm}
Recall that under these hypotheses $\Lmu(A,B)$ is a model for the derived relative smash product in the homotopy category of $R$-modules.

\begin{proof}
	We apply \Cref{prop: alg structure over BP} to both $A$ and $B$ obtaining the maps
	\[
	\Lmu(A,A) \lra A \ \text{and} \ \Lmu(B,B)\lra B
	\]
	satisfying the appropriate associativity condition.
	Then we lift these structures to maps of filtrations using \Cref{lemma: lift of product}.
	This produces maps of filtrations
	\[
	\Lmb(\Ab,\Ab) \lra \Ab
	\]
	where $\Ab$ is projective and exact filtrations of $A$ as an $R$-module.
	These exist by the construction given in Section 2.2 of \cite{KSSpaper}.
	The filtration $\cbu{B}$ is also multiplicative.
	Then using \Cref{prop: lambda filt mult} we have that $\Lmb(\Ab,\cbu{B})$ is multiplicative.
	Thus the spectral sequence is multiplicative by \Cref{prop: still mult ss}.
\end{proof}

\subsection{The K\"unneth spectral sequence over \texorpdfstring{$BP$}{BP}.}

In order to apply the above results, we need maps of associative $S$-algebras
\[
BP \lra A.
\]
When $A$ is $\HF_p$ this is clear as we can take the composition of $0$th Postnikov sections and the reduction mod $p$ map
\[
BP \lra P_0(BP)=\HH\Z_{(p)} \lra \HF_p
\]
which are both maps of associative algebras.
Obtaining the map to connective Morava $\moravaE$-theory is more difficult and relies on the work of Lazarev in \cite{LazarevTowersMU}, Angeltveit in \cite{AngeltveitObstruction}, and Rognes in \cite{RognesGalois}.
Once we have a map of associative $S$-algebras
\[
BP \lra \moravaE_n
\]
we get a map to $e_n$ since $BP$ is connective.
The following result is not original and is proved by stitching together various results in the literature.
We don't know of a proper reference and so we provide the argument here.

\begin{lemma}
	\label{lemma:map to morava}
	There is a map of associative $S$-algebras
	\[
	BP \lra e_n
	\]
	which takes the class $v_i$ to $u_iu^{p^i-1}$ for $i<n$, $u^{p^n-1}$ for $i=n$, and $0$ otherwise.
\end{lemma}
\begin{proof}
	There are maps of associative $S$-algebras
	\[
	BP \stackrel{f_1}{\lra} E(n) \stackrel{f_2}{\lra} \widehat{E(n)} \stackrel{f_3}{\we} \moravaE_n^{hK} \stackrel{f_4}{\lra} \moravaE_n,
	\]
	where
	\begin{itemize}
		\item $E(n)$ denotes periodic Johnson-Wilson theory and $f_1$ is the map of associative $S$-algebras constructed by Lazarev in \cite{LazarevTowersMU} and Angeltveit in \cite{AngeltveitObstruction},
		\item $\widehat{E(n)}$ is the $I$-adic completion and $f_2$ is the natural map, which is a map of associative $S$-algebras.
		\item The map $f_3$ is constructed by Rognes in his proof of \cite[Proposition 5.4.9]{RognesGalois}. Here, $K$ is a particular subgroup of the extended Morava stabilizer group.
		\item The map $f_4$ is also constructed by Rognes in \cite{RognesGalois}. It follows from his computation that these are maps of commutative $S$-algebras and that the composite $f_4 \circ f_3 \circ f_2 \circ f_1$ takes the class $v_i$ to the desired element in $\pi_*(\moravaE_n)$.
	\end{itemize}
	Composing these maps yields a map $BP \lra \moravaE_n$ of associative $S$-algebras.
	It is lifted to the connective cover $e_n$ as follows.
	We can construct $e_n$ as an associative $S$-algebra by attaching cells to $\moravaE_n$ to kill all of the nonnegative homotopy groups in the category of associative $S$-algebras.
	This gives a map
	\[
	\moravaE_n \lra \tau_{<0} \moravaE_n
	\]
	of associative $S$-algebras which is an isomorphism in $\pi_i$ for $i<0$.
	The fiber of this is then connective and an associative $S$-algebra as limits in $S$-algebras are computed in the underlying category of $S$-modules.
	The spectrum $BP$ is connective, therefore the composition
	\[
	BP \lra \moravaE_n \lra \tau_{<0} \moravaE_n
	\]
	is nullhomotopic.
	Thus we have our desired lift $BP \lra e_n$ in the category of $S$-algebras.
\end{proof}

We are now in a situation to apply the above result \Cref{thm: Comparison Theorem} to the maps of $S$-algebras
\[
BP \lra \HF_p,\ \ BP \lra \HF_p\wedge BP, \ \ \text{and} \ \ BP \lra e_n.
\]

\begin{cor}
	The K\"unneth spectral sequences 
	\[
	\Tor^{BP_*}_s (\HH_*(BP;\F_p),e_{n*})_t \ssra \HH_{s+t}(e_n;\F_p) \ \text{and} \ \Tor^{BP_*}_s (\F_p,e_{n*})_t \ssra \pi_{s+t}(\HF_p\wedge_{BP} e_n)
	\]
	are multiplicative.
\end{cor}
\begin{proof}
	We apply \Cref{thm: mult of KSS of E_3 alg} to $\HF_p, \HF_p \wedge BP,$ and $e_n$.
	Note that the application of \Cref{prop: alg structure over BP} to $\HF_p\wedge BP$ is slightly subtle.
	We obtain the map
	\[
	\Lmu(\HF_p\wedge BP, \HF_p \wedge BP) \lra \HF_p \wedge BP
	\]
	satisfying the appropriate coherences by noting that there is always a natural map
	\[
	\Lf(\HF_p\wedge BP,\ldots, \HF_p \wedge BP) \lra \Lf(\HF_p,\ldots, \HF_p )\wedge \Lf(BP,\ldots,BP)
	\]
	using the diagonal map on the domain of $f$ and the fact that $\Lf(BP,\ldots,BP)\we BP\wedge X_+$ for all $f$.
	The rest of the proof of \Cref{thm: mult of KSS of E_3 alg} applies directly.
\end{proof}

\subsection{Massey Products in the K\"unneth spectral sequence}
\label{subsec: massey}
Here we review some material from \cite{BakerRichterAdams}.
The results of this section are due to Baker, Richter, and Kochman.
Our only contribution is the observation that they apply in our setting.
Kochman establishes conditions under which elements of Massey products detect elements of Toda brackets in the Adams spectral sequence in \cite{KochmanBook}.
Then in \cite{BakerRichterAdams}, Baker and Richter adapt this to multiplicative K\"unneth spectral sequences.
This then relates Massey products in $\Tor$ to Toda brackets in the target.
We recall briefly some of the relevant definitions and then we will state their Theorem B.2 from \cite{BakerRichterAdams}.

Recall the convention that $\h{a}=(-1)^{\vert a\vert +1}$ where $\vert a\vert$ is the total degree of the homology class $a$, see Appendix 1 of \cite{GreenBook} where the notation $\oo{a}$ is used instead.
We also use the notation $[a]$ to denote the homology or homotopy class of a cycle $a\in A$ or a map
\[
a:S^n \lra E
\]
depending on the context.

\begin{defn}
	Let $[x],[y],[z]\in \HH_*(A)$ where $A$ is a DGA and $[x][y]=0$, $[y][z]=0$ in $\HH_*(A)$.
	Then the \Defn{Massey product} $\<[x],[y],[z]\>$ is defined to be the set $ \{\h{s}z+\h{x}t\vert \partial(s)=\h{x}y$ and $\partial(t)=\h{y}z\}$.
\end{defn}

The data $\{s,t,x,y,z\}$ along with their boundaries is called a defining system.
The indeterminacy of this Massey product is given by
\[
[x]\HH_* (A)\dsum  \HH_* (A)[z]
\]
for suitable degrees.

For homotopy groups of ring spectra there is the similar notion of Toda brackets.
In the presence of a product structure the definitions are very similar.
Therefore it is not surprising that they are related.

\begin{defn}
	Let $[a],[b],[c]\in \pi_*E$ where $E$ is an $S$-algebra and $[a][b]=0$, $[b][c]=0$ in $\pi_*E$.
	Then the \Defn{Toda bracket} $\<[a],[b],[c]\>$ is defined to be the set of all elements of the form $g_{ab}c+ag_{bc}$, where 
	\[
	g_{ij}: D^{\vert i \vert +\vert j \vert +1} \lra E
	\]
	is a nullhomotopy of the product 
	\[
	i\wedge j: S^{\vert i\vert + \vert j \vert}\iso S^{\vert i \vert} \wedge S^{\vert j \vert} \lra E \wedge E \sr{\mu} \lra E.
	\]
\end{defn} 
Note that we only use the existence of a product on $E$.
This works in modules over any $S$-algebra, such as $BP$.

We now have the result of Baker and Richter which allows us to compute differentials by relating Massey products and Toda brackets.

\begin{thm}[Theorem B.2 of \cite{BakerRichterAdams}]
	\label{thm: Baker Richter B.2}
	Assume that the following conditions hold in the KSS
	\[
	\Tor^{R_*}(A_*,B_*)\ssra \pi_*(A\wedge_R B).
	\]
	\begin{itemize}
		\item The elements $x,y,z\in \Page{r}$ are permanent cycles which converge to elements $\xi_1, \xi_2, \xi_3$ in $\pi_*(A\wedge_R B)$ respectively.
		\item The Massey product $\<[x],[y],[z]\>$ is defined in $\Page{r+1}$.
		\item The Toda bracket $\<\xi_1,\xi_2,\xi_3\>$ is defined in $\pi_*(A\wedge_R B)$.
		\item If $\{s,t,x,y,z\}$ is a defining system for $\<[x],[y],[z]\>$ then there are no non-zero crossing differentials for the differentials $d_r(s)=\h{x}y$ and $d_r{t}=\h{y}z$.
	\end{itemize}
	Then $\<[x],[y],[z]\>$ is a set of permanent cycles which converge to elements of $\<\xi_1,\xi_2,\xi_3\>$.
\end{thm}
Even though Baker and Richter require commutativity, it is not necessary.
They only use commutativity of the base to establish that the KSS is multiplicative.
Given our \Cref{thm: mult of KSS of E_3 alg}, all that is necessary is an $\Ealg_3$-algebra.
Further, the only part of multiplicativity used in the proof is the pairing of the filtration with itself.

The notion of crossing differential is slightly technical.
Let $y\in \Page{r}_{a,n-a}$ with $d_r y \neq 0$.
We say that $y$ has a \Defn{crossing differential} if there is an element $y' \in \Page{r'}_{a',n-a'}$ with $d_{r'} y' \neq 0$ such that $a < a'$ and $a + r > a' + r'$.
If we draw a spectral sequence we see that these differentials cross each other.
In our situation, there will be no non-zero crossing differentials for degree reasons.
This result will allow us to show that certain classes contained in Massey products on the $E_2$-page survive the spectral sequence to detect genuine homotopy classes.

\section{The K\"unneth spectral sequence for \texorpdfstring{$\HH_*(e_n;\F_p)$}{H\_*(e\_n;F\_p)}}
\label{sec: E_2 page}

In this section we compute the $\Page{2}$-page of the K\"unneth spectral sequence
\[
\Page{2}_{s,t} = \Tor^{BP_*}_s(\HH_*(BP;\F_p),e_{n*})_t \ssra \HH_{s+t}(e_n;\F_p)
\]
as an algebra.
We then use \Cref{thm: Baker Richter B.2} to show that the spectral sequence collapses for $n\leq 4$.
This is enough to compute the homology of connective Morava $\moravaE$-theory up to multiplicative extensions.
We resolve some of the multiplicative extensions in \Cref{subsec: extension problems}.

\subsection{The \texorpdfstring{$\Page{2}$}{E\^{}2}-page as an algebra}
In order to compute $\Tor^{BP_*}_*(\HH_*(BP;\F_p),e_{n*})$, we first recall the descriptions of the algebras involved:
\begin{align*}
	BP_* & \iso \Z_{(p)}[v_1,v_2,\ldots],\\
	\HH_*(BP;\F_p) &\iso 
	\begin{cases}
		\F_2[\xi_1^2,\xi_2^2,\ldots] &\text{ if } p = 2;\\
		\F_p[\xi_1,\xi_2,\ldots] &\text{ if } p \neq 2,
	\end{cases}\\
	e_{n*} &\iso \W(\F_{p^n})[[u_1,u_2,\dots,u_{n-1}]][u]
\end{align*}
The generators are graded as follows:
\begin{align*}
	\deg(v_i) &= 2(p^i-1) & i\geq 1,&\\
	\deg(\xi_i^2) &= 2(2^i-1) & i\geq 1,& \text{ for } p=2,\\
	\deg(\xi_i) &= 2(p^i-1) & i \geq 1,& \text{ for } p\neq 2,\\
	\deg(u_i) &= 0,&\\
	\deg(u) & = 2.
\end{align*}
See \cite{GreenBook} for $\HH_*(BP;\F_p)$ and $BP_*$.
Here, $\W(\F_{p^n}) = \Z_p[\zeta_n]$ is the ring of $p$-typical Witt vectors over the field $\F_{p^n}$, where $\ZZ_p$ denotes the $p$-adic integers and $\zeta_n$ is a primitive $(p^n-1)^\mathrm{st}$ root of unity.

$\HH_*(BP;\F_p)$ is an algebra over $BP_*$ via the Hurewicz map.
This map sends $p$ and each $v_i$ to zero, so as an $BP_*$-module, $\HH_*(BP;\F_p)$ is isomorphic to $\bigoplus \F_p$.
The $BP_*$-algebra structure on $e_{n*}$ is induced from the map constructed in \Cref{lemma:map to morava}, which maps
\[ 
	v_i \mapsto 
	\begin{cases}
	u^{p^i-1}u_i &\text{ if } 1 \leq i \leq n-1\\
	u^{p^i-1}&\text{ if } i = n\\
	0 &\text{ if } i > n.
	\end{cases}
\]

\begin{lemma}\label{lemma: reduction}
	$\Tor^{BP_*}_*(\HH_*(BP;\F_p),e_{n*}) \iso \HH_*(BP;\F_p) \tensor_{\F_p} \Tor^{\ZZ_{(p)}[v_1,\dotsc, v_n]}_*(\F_p,e_{n*}) \tensor_{\F_p} \ext{\F_p}{\ov_{n+1},\ov_{n+2},\ldots}$.
\end{lemma}
Note that $\ZZ_{(p)}[v_1,\dotsc, v_n] \iso \BPn_*$, so the remaining part $\Tor^{\ZZ_{(p)}[v_1,\dotsc, v_n]}_*(\F_p,e_{n*})$ can be interpreted in terms of $\BPn$.
\begin{proof}
	Since $\HH_*(BP;\F_p)\iso \bigoplus \F_p$ as a $BP_*$-module, we have the isomorphism
	\[ \Tor^{BP_*}_*(\HH_*(BP;\F_p),e_{n*}) \iso \HH_*(BP;\F_p) \tensor_{\F_p} \Tor^{BP_*}_*(\F_p,e_{n*}). \]
	Next, we compute $\Tor^{BP_*}_*(\F_p,e_{n*})$ by resolving $\F_p$ over $BP_*$ using the Koszul complex $K(p, v_1, \dotsc)$.
	We write $\ov_i$ for the element with $\partial(\ov_i) = v_i$.
	Since for $i > n$, $\ov_i$ is a cycle in $K(p, v_1, \dotsc) \otimes_{BP_*} e_{n*}$, we obtain the isomorphism
	\[
	\Tor^{BP_*}_*(\F_p,e_{n*}) \iso \Tor^{\ZZ_{(p)}[v_1,\dotsc, v_n]}_*(\F_p, e_{n*}) \tensor_{\F_p} \ext{\F_p}{\ov_{n+1},\ov_{n+2},\ldots}. \qedhere
	\]
\end{proof}

We will compute $\Tor^{\ZZ_{(p)}[v_1,\dotsc, v_n]}_*(\F_p,e_{n*})$ using the Koszul complex $K(p, v_1, \dotsc, v_n)$.
We use the convention that $v_0=\set{0, 1, \dotsc, n}$, $u_0 = p$ and $u_n = 1$.
Then we have that
\[ \Tor^{\ZZ_{(p)}[v_1,\dotsc, v_n]}_*(\F_p,e_{n*}) = H_*(\sK) \]
where $\sK := K(p, v_1, \dotsc, v_n) \otimes_{\ZZ_{(p)}[v_1,\dotsc, v_n]} e_{n*} = K(u_0 u^{w(0)}, u_1 u^{w(1)}, \dotsc, u_n u^{w(n)})$, and $w(i) = p^i-1$.

In order to describe $H_*(\sK)$, we need to set up some notation.
Recall that $[n] = \set{1, \dotsc, n}$ for $n \in \NN$.
In addition, we set $[n]_0 := [n] \cup \set{0}$.
For $I = \set{i_1, \dotsc, i_r} \subset [n]_0$ with $i_1 < \dotsb < i_r$ we set $\ov_I := \ov_{i_1} \wedge \dotsm \wedge \ov_{i_r}$ and $\alpha_I := u^{w(i_1)}$, extending the notation $\ov_i$.
\begin{lemma}\label{lem:generators}
	The homology of $\sK$ is generated by $1$ and the homology classes of
	\[ f_I := \frac{1}{\alpha_I}\partial(\ov_I) \]
	for $I \subseteq [n]$, $\#I \geq 2$, as an $e_{n*}$-module.
\end{lemma}
\noindent Note that we only need $f_I$ with $0 \notin I$ as generators.
\begin{proof}
    First, we show that the $f_I$ with $I \subseteq [n]_0$ generate the cycles of $\sK$.
	For this, we need to show that the elements $f_I$ are well-defined.
	For this, consider a set $I \subseteq [n]_0$ and set $i := \min I$.
	Since $w$ is monotonic it holds that $u^{\w{i}} \mid \partial \ov_k$ for all $k \in I$, and thus $f_I$ is well-defined.
	Moreover, each $f_I$ is a cycle, because $0 = \partial\partial \ov_I = \alpha_I \partial f_I$ and $\alpha_I$ is a non-zero divisor on $\sK$.
	
	It remains to show that every cycle in $\sK$ can be written as a linear combination of the $f_I$'s.
	We will use the lexicographic order on the power set of $[n]_0$ to prove this.
	For any element $c = \sum_{J \subseteq [n]_0}c_J \ov_J \in \sK$ we call $\supp(c) := \set{J \subseteq [n]_0 \with c_J \neq 0}$ its support.
	Moreover, the \emph{leading term} of $c$ is $c_J \ov_J$ for $J = \max(\supp(c))$.

	Consider a cycle $c:=\sum_{J}c_J \ov_J \in \sK$.
	Let $J_0 := \max(\supp(c))$ and $j_0 := \min(J_0)$.
	The $\ov_J$ are linearly independent, so we can consider the coefficient of $\ov_{J_0 \setminus {j_0} }$ in $\partial c$ to obtain that
	\[ 
	c_{J_0} \partial\ov_{j_0} + \sum_{j \in [n]_0 \setminus J_0} \pm c_{(J_0 \setminus {j_0}) \cup {j}} \partial\ov_{j}  = 0 
	\]
	By the definition of $J_0$ and $j_0$, only the coefficients $c_{(J_0 \setminus {j_0}) \cup {j}}$ with $j < j_0$ are non-zero.
	Hence $c_{J_0} \partial\ov_{j_0} = c_{J_0} u^{\w{j_0}} u_{j_0}$ is contained in the ideal $(u_j \with 1 \leq j < j_0) \subset e_{n*}$.
	As $u_{j_0}$ is a non-zero divisor modulo this ideal or a unit and $u$ is an indeterminate, it follows that already $c_{J_0}$ is contained in the ideal.
	Hence there is a presentation of $c_{J_0}$ as $c_{J_0} = \sum_{1 \leq j < j_0} s_j u_j$ with $s_j \in e_{n*}, 1 \leq j < j_0$.
	Consider the cycle
	\[ 
	c_1 := \sum_{1 \leq j < j_0} s_j f_{J_0 \cup {j}}. 
	\]
	Since $j < j_0 = \min(J_0)$, the leading term of $c_1$ is
	\[ 
	\sum_{1 \leq j < j_0} s_j u_j \ov_{(J_0 \cup {j})\setminus{j}} = \left(\sum_{1 \leq j < j_0} s_j u_j\right) \ov_{J_0} = c_{J_0} \ov_{J_0}. 
	\]
	Hence $c' := c - c_1$ is a cycle with a strictly smaller leading term than $c$.
	As there are only finitely many sets $I \subseteq [n]_0$, the claim follows by induction.
	
	Finally, note for $I \subseteq[n]_0$ with $0\in I$, it holds that $\alpha_I = 1$ and thus $f_I = \partial(\ov_I)$ is a boundary.
	Hence, we only need $f_I$ with $I \subseteq [n]$ to generate the homology of $\sK$.
\end{proof}

Let $A := \ext{e_{n*}}{f'_I \with I \subseteq [n], \#I \geq 2}$ be the exterior $e_{n*}$-algebra with the indicated generators.
We endow $A$ with a bigrading by setting 
$\deg f'_I := (\#I - 1, \sum_{i \in I \setminus \min(I)} 2w(i))$.
Here, the first component of the grading is to be interpreted as a homological grading, and the second component is an internal grading.
In particular, the commutativity relation of $A$ is
\[ f'_I f'_J = (-1)^{(\#I-1)(\#J-1)} f'_J f'_I.\]

By the previous lemma, we have a map
\begin{align*}
	\psi\colon A &\longrightarrow \sK\\
	f'_I &\longmapsto f_I
\end{align*}
of $e_{n*}$-algebras, whose image is the cycles of $\sK$.
It induces a surjective map $[\psi]\colon A \to H_*(\sK)$.

\begin{lemma}\label{lem:relations}
	The kernel of $[\psi]$ is generated by the following polynomials:
	\begin{align*}
	u^{w(i)} u_i &\quad\text{for } 0 \leq i \leq n;\\
	u^{w(\min I)} f'_I &\quad \text{for } I \subseteq [n], \#I \geq 2;\\
	u_a f'_{I \cup b} - u_b f'_{I \cup a} &\quad \text{for } I \subseteq [n],\#I\geq2,\  a,b \in [n] \text{ with } a,b < \min(I)
	\end{align*}
	and 
	\[
	f'_I\cdot f'_J -
	\begin{cases}
	(-1)^{m(I \setminus i_0, J)} u_{i_0} f'_{(I \setminus i_0) \cup J} &\text{ if } i_0\geq j_0 \text{ and } (I \setminus i_0) \cap J  = \emptyset; \\
	(-1)^{m(I, J \setminus j_0)} u_{j_0} f'_{I \cup (J \setminus j_0)} &\text{ if } j_0\geq i_0 \text{ and } I \cap (J\setminus j_0) = \emptyset; \\
	0 & \text{ otherwise} 
	\end{cases}
	\]
	for all $I, J \subset [n],\#I\geq2,\#J \geq2$, where $i_0 := \min(I)$, $j_0 := \min(J)$, and $m(A,B) := \#\set{(a,b)\in A\times B\with a > b}$.
\end{lemma}

\begin{proof}
	Let $\af \subset A$ denote the ideal with the given generators.
	First, we show that $\af \subseteq \ker([\psi])$.
	For this we need to show that the classes of the $f_I$ satisfy the given relations.
	It is clear that $u^{w(i)} u_i = \partial \ov_i$ and $u^{w(\min I)} f_I = \partial \ov_i$ are boundaries.
	Next, we show that $u_a f_{I \cup b} - u_b f_{I \cup a}$ is a boundary for $I \subseteq [n]$, $a,b \in [n]$ with $a,b < \min(I)$.
	For this, we compute
	\[\begin{aligned}
	0 &= \partial \partial \ov_{\set{a,b} \cup I} = \partial\left(\partial\ov_a \ov_{{b} \cup I} - \partial\ov_b \ov_{{a} \cup I} + \sum_{i \in I} \partial \ov_i \ov_{\set{a,b} \cup I\setminus{i}}\right) \\
	&= \partial\ov_a \partial \ov_{{b} \cup I} - \partial\ov_b \partial \ov_{{a} \cup I} + \sum_{i \in I} \partial \ov_i \partial\ov_{\set{a,b} \cup I\setminus{i}}\\
	&= u_a u^{\w{a}} u^{\w{b}} f_{{b} \cup I} - u_b u^{\w{b}} u^{\w{a}} f_{{a} \cup I} + \sum_{i \in I} u_i u^{w(i)} \partial\ov_{\set{a,b} \cup I\setminus{i}} \\
	& u^{\w{a}} u^{\w{b}} \left(u_a f_{{b} \cup I} - u_b f_{{a} \cup I} + \sum_{i \in I} u_i u^{\w{i} - \w{a} - \w{b}} \partial\ov_{\set{a,b} \cup I\setminus{i}} \right).
	\end{aligned}\]
	An elementary computation shows that $\w{i} - \w{a} - \w{b} \geq 0$ if $i > a,b$.
	Thus we have that 
	\begin{equation*}
	u_a f_{{b} \cup I} - u_b f_{{a} \cup I}  = - \sum_{i \in I} u_i u^{\w{i} - \w{a} - \w{b}} \partial\ov_{\set{a,b} \cup I\setminus{i}}
	\end{equation*}
	is a boundary.
	
	Next, we show that there are no other linear generators of $\ker([\psi])$.
	This is very similar to our argument in the proof of \Cref{lem:generators} above.
	In homological degree $0$, $\af$ clearly contains all the preimages of the boundaries.
	In higher homological degree, consider an element $r := \sum_{J}c_J f'_J \in \ker([\psi])$.
	Then $\psi(r)$ is a boundary, so there exists an element $c = \sum_{J'} c'_{J'} \ov_{J'}$ such that $\partial c = \psi(r)$.
	In other words,
	\begin{equation}\label{eq:relF}
	0 = \sum_{J}c_J f_J - \sum_{J'} c'_{J'} \partial\ov_{J'} = \sum_{J}c_J f_J - \sum_{J'} c'_{J'} \alpha_{J'} f_{J'}
	= \psi\left(\sum_{J}c_J f'_J - \sum_{J'} c'_{J'} \alpha_{J'} f'_{J'}\right).
	\end{equation}
	As $\alpha_{J'} f'_{J'} \in \af$ for all $J'$, we may replace $r$ by $\sum_{J}c_J f'_J - \sum_{J'} c'_{J'} \alpha_{J'} f'_{J'}$ and thus assume that $\psi(r) = 0$ as a cycle.
	
	Next, let $J_0$ be the (lexicographically) largest set in $supp(r)$, and let $j_0 := \min(J_0)$.
	Considering the coefficient of $\ov_{J_0 \setminus j_0}$ in $\psi(r)$, we see that 
	\[ 
	c_{J_0} u_{j_0} + \sum_{j \in [n] \setminus J_0} c_{(J_0 \setminus {j_0}) \cup {j}} u_j  = 0.
	\]
	Now we argue as above that $c_{J_0}$ is contained in the ideal $(u_j \with 1 \leq j < j_0) \subset e_{n*}$ and can thus be written as $c_{J_0} = \sum_{1 \leq j < j_0} s_j u_j$ with $s_j \in e_{n*}, 1 \leq j < j_0$.
	Consider the element
	\[ 
	r_1 := \sum_{1 \leq j < j_0} s_j \Big(u_j f'_{J} - u_{j_0} f'_{{j} \cup J \setminus\set{j_0}}\Big) \in \af,
	\]
	As before, $r$ and $r_1$ have the same leading term, so we may replace $r$ by $r - r_1$.
	By induction, it follows that $r$ can be written as a sum of terms of this form, and thus $r \in \af$.
	
	Finally, we show the formula for the product $f_I f_J$.
	For this let $I, J \subseteq [n]$, $\#I \geq 2, \#J \geq 2$ with $i_0 := \min(I), j_0 := \min(J)$. By symmetry we may assume that $i_0 \leq j_0$. We start with a short computation:
	\begin{align*}
	f_I\cdot f_J &= \frac{1}{\alpha_I \alpha_J} (\partial \ov_I)(\partial \ov_J)
	= \frac{1}{\alpha_I \alpha_J} \partial(\ov_I \cdot \partial \ov_J)
	= \frac{1}{\alpha_J} \sum_{j \in J} \partial(\ov_j) \frac{1}{\alpha_I} \partial(\ov_{I} \cdot \ov_{J\setminus j})\\
	&= \sum_{j \in J} u_j u^{\w{j} - \w{j_0} - \w{i_0}} \partial(\ov_{I} \cdot \ov_{J\setminus j}).
	\end{align*}
	For $j > j_0$ we have that $\w{j} - \w{j_0} - \w{i_0} \geq 0$, because $i_0 \leq j_0$.
	Thus, the only term which is possibly not a boundary is the term with $j = j_0$.
	If $I \cap J\setminus{j_0} \neq \emptyset$, then $\ov_I \cdot \ov_{J\setminus {j_0}} = 0$.
	Otherwise, we have that
	$u_{j_0} u^{\w{j_0} - \w{j_0} - \w{i_0}} \partial(\ov_{I} \cdot \ov_{J\setminus j}) = u_{j_0} f'_{I \cup J \setminus j}$.
	Hence the product of $f_I$ and $f_J$ is as claimed.
\end{proof}

Altogether, we have proven the following result:
\begin{thm}\label{thm:E2}
	It holds that
	\[
		\Tor^{BP_*}_*(\HH_*(BP;\F_p),e_{n*}) \iso
		\HH_*(BP;\F_{p}) \tensor_{\F_{p}} A/\af \tensor_{\F_{p}} \ext{\F_{p}}{\ov_{n+1},\ov_{n+2},\ldots}
	\]
	as bigraded $e_{n*}$-algebras, where $\af$ is the ideal generated by the polynomials given in \Cref{lem:relations}, and the right-hand side is graded as follows:
	\begin{itemize}
		\item $\HH_*(BP;\F_p)$ in concentrated in homological degree $0$ and has internal degrees equal to the total degrees,
		\item and $\ov_{n+k}$ for $k \geq 1$ has homological degree $1$ and internal degree $2(p^{n+k}-1)$
	\end{itemize}
\end{thm}

In addition to its algebra structure, $\Tor^{BP_*}_*(\HH_*(BP;\F_p),e_{n*})$ also has Massey products. 
See \Cref{subsec: massey} for definitions and conventions regarding Massey products.
The following result is crucial for our application.
It was inspired by \cite[Proposition 5.3]{BakerRichterAdams}.

\begin{prop}
	\label{prop:massey for f_I}
	Let $I, J \subseteq [n]$ be two disjoint sets with $\min(J) > \min(I)$.
	Then
	\[ 
		(-1)^{\#I + m(I, J)} f_{I \cup J}\in \<f_I,u^{\w{\min J }},f_J\> 
	\]
	 in $\Tor^{BP_*}_*(\HH_*(BP;\F_p),e_{n*})$.
\end{prop}
Recall the notation $m(I,J) = \#\set{(i,j)\in I\times J\with i > j}$.
\begin{proof}
	Recall our convention that $\hat{r} = (-1)^{|r| + 1}r$, where $r \in \Tor^{BP_*}_*(\HH_*(BP;\F_p),e_{n*})$ and $|r|$ denotes its total degree.
	However, all elements have even internal degree we may use the homological degree instead.
	Let $i_0 := \min(I)$ and $j_0 := \min(J)$.
	Note that $\hat{f}_I u^{\w{j_0}} = (-1)^{\#I}\partial(u^{\w{j_0}-\w{i_0}}\ov_{I})$ and $\hat{u}^{\w{j_0}} f_J = (-1)^{0+1}\partial\ov_J$.
	Hence the class of the element 
	\begin{align*}
	 ((-1)^{\#I}u^{\w{j_0}-\w{i_0}}\ov_{I})^{\widehat{{}}}\cdot f_J + (-1)^{\#I+1} f_I \cdot \ov_J &= 
		- u^{\w{j_0}-\w{i_0}}\ov_{I} \cdot f_J + (-1)^{\#I+1} f_I \cdot \ov_J\\
		&=- \ov_I \cdot \frac{1}{\alpha_I} \partial(\ov_J) + (-1)^{\#I+1} \frac{1}{\alpha_I} \partial(\ov_I) \cdot \ov_J \\
		&= (-1)^{\#I} \frac{1}{\alpha_I} \Big( \partial(\ov_I) \cdot \ov_J + (-1)^{\#I} \ov_I \cdot \partial(\ov_J)\Big) \\
		&= (-1)^{\#I}  \frac{1}{\alpha_I} \partial(\ov_I\ov_J)\\
		&= (-1)^{\#I + m(I, J)} f_{I \cup J}
	\end{align*}
	is contained in the Massey product $\<f_I,u^{\w{j_0}},f_J\>$.
\end{proof}

\subsection{The collapse of the spectral sequence}
We are now able to prove our main theorem.

\begin{thm}\label{thm: collapse}
	The spectral sequence
	\[
	\Tor^{BP_*}_s(\HH_*(BP;\F_p), e_{n*})_t \ssra \HH_{s+t}(e_n;\F_p)
	\]
	collapses at the $\Page{2}$-page when $n\in \{1,2,3,4\}$.
\end{thm}
Our argument follows the proof of Theorem 7.3 in \cite{BakerRichterAdams} by Baker and Richter.
The case when $n=1$ is classical as $e_1=ku_p$, the $p$-completion of the connective complex $K$-theory spectrum.

\begin{proof}
	There are $4$ different types of multiplicative generators in our $\Page{2}$-page, namely
	the elements contributed by $\HH_*(BP;\F_p)$, the generators of $e_{n*}$, the $\ov_i$ for $i > n$, and finally the $f_I$.
	There are elements coming from $\HH_*(BP;\F_p)$ and $e_{n*}$ that are on the $0$-line and are necessarily permanent cycles.
	Similarly, the $\ov_i$ are on the $1$-line and thus permanent cycles.
	It remains to show that the $f_I$ are permanent cycles.
	This will establish the collapse of the spectral sequence at the $\Page{2}$-page.
	
	The proof follows by induction on total degree.
	First, $\Page{2}=\Page{3}$ as $d_2$ increases the internal degree by $1$ and every element in the $\Tor$ group has even internal degree.
	This then implies that the $f_{i,j}$ and the $f_{i,j,k}$ are permanent cycles as they are on the $1$- and $2$-line, respectively, of the spectral sequence and the only remaining differentials they could support decrease homological degree by at least $3$.
	Further, the relation $\alpha_{i,j}f_{i,j}=0$ persists to $\HH_*(e_n;\F_p)$.
	This is because there are no elements of odd total degree on the $0$-line.
	Thus there can not be an element of lower filtration and same total degree other than $0$.
	
	From this we can deduce that $f_I$ where $I=\{i_1,i_2,i_3,i_4\}$ is a permanent cycle as follows.
	By \Cref{prop:massey for f_I} we have that $f_I\in \<f_{i_1,i_2},u^{p^{i_3}-1},f_{i_3,i_4}\>$.
	The indeterminacy of this Massey product consists of permanent cycles as they are decomposable with respect to the product structure.
	The classes $f_{i_1,i_2}$, $u^{p^{i_3}-1}$, and $f_{i_3,i_4}$ are each permanent cycles which detect homotopy classes.
	We can form the Toda bracket of these homotopy classes since $\alpha_{i,j}f_{i,j}=0 \in \HH_*(e_n;\F_p)$.
	By \Cref{thm: Baker Richter B.2}, we have that the element $f_I$ in the $\Page{2}$-page detects an element in the Toda bracket
	$\<f_{i_1,i_2},u^{p^{i_3}-1},f_{i_3,i_4}\>$ as long as their are no non-zero crossing differentials.
	This is indeed the case as the domains of the possible crossing differentials are in lower total degree and so must be trivial.
	Thus $f_I$ detects an element in $\HH_*(e_n;\F_p)$ as desired.
\end{proof}

The above argument shows that in fact $f_I$ detects an element in $\HH_*(e_n; \F_p)$ for all $n$ when $\#I \leq 4$.
It seems unlikely that this approach can be pushed further as we have no way of showing that the product $\alpha_I f_I$ is not divisible by an element of $\HH_*(BP;\F_p)$ in general.
This would be possible if we had natural maps from $e_n$ to $e_{n+1}$, but we know of no such maps.
However, we are able to say something about spectrum $\pi_*(\HF_p \wedge_{BP} e_5)$.

\begin{prop}
	The spectral sequence 
	\[
	\Tor^{BP_*}_s(\F_p, e_{5*})_t \ssra \pi_{s+t}(\HF_p \wedge_{BP}e_5)
	\]
	collapses at the $\Page{2}$-page.
\end{prop}

\begin{proof}
	The argument above works to establish that everything is a permanent cycle with the exception of the element $f_I$ where $I=\{1,2,3,4,5\}$.
	This will follow from the fact that $u^{p^3-1}f_{3,4,5}=0$ in $\pi_*(\HF_p \wedge_{BP}e_5)$.
	The bidegree of $u^{p^3-1}f_{3,4,5}$ is $(2,2p^3-2+2p^4-2+2p^5-2)$ and it has total degree $2p^3+2p^4+2p^5-4$.
	This product could be non-zero in the target if there were an element in lower filtration and the same total degree.
	The only elements in lower filtration and positive total degree are multiples of $u$.
	In filtration $0$ we have $u_i$ multiples of powers of $u$.
	In order to reach that total degree the exponent of $u$ must be at least $p^5$, but $u^{p^5-1}=0$.
	In filtration $1$ we have $f_{i,j}$ and these are all of odd degree so no product of an $f_{i,j}$ and a power of $u$ could have the right total degree.
	Therefore we have that $u^{p^3-1}f_{3,4,5}$ is $0$ in the target of the spectral sequence and not just in the $\Page{\infty}$-page.
	Now we use \Cref{prop:massey for f_I} and \Cref{thm: Baker Richter B.2} to show that $f_I$ is a permanent cycle as in the proof of \Cref{thm: collapse}.
\end{proof}

While this is not a direct computation regarding the homology of $e_5$, it does give us a lot of information since the $\Page{2}$-page splits as a tensor product by \Cref{lemma: reduction}.
This relative smash product $\pi_*(\HF_p \wedge_{BP} e_5)$ still contains all of the new interesting classes $f_I$.

Note that these results also imply the following.
\begin{cor}
	The spectral sequence 
	\[
	\Tor^{\BPn_*}_s(\F_p,e_{n*})_t \ssra \pi_{s+t}(\HF_p\wedge_{\BPn}e_n)
	\]
	collapses for $n\leq 5$.
\end{cor}
This follows as the map of associative $S$-algebras $BP \lra \BPn$ induces a map of spectral sequences.
This allows us to compute the differentials on all classes in the target by computing them in the source since the map is surjective on $\Tor$.
This fact will be used in the next section.

\subsection{Multiplicative extensions}
\label{subsec: extension problems}
In this section we show that many relations of the form $xy=0$ in the $\Page{\infty}$-page of the spectral sequence
\[
\Tor^{\BPn_*}_s(\F_p,e_{n*})_t \ssra \pi_{s+t}(\HF_p\wedge_{\BPn}e_n)
\]
in fact hold in homotopy as well.
After this we establish that $\HH_*(e_n;\F_p)$ splits into a tensor product of rings, see \Cref{cor: mult splitting}.
We will use the collapse of the above spectral sequences established in \Cref{thm: collapse}. Therefore we assume $n\leq 4$ throughout this section.

\begin{lemma}
	\label{prop: mult extensions}
	We have the following relations in the ring $\pi_*(\HF_p\wedge_{\BPn} e_n)$:
	\begin{enumerate}
		\item $u_i u^{p^i-1}=0$ and $u^{p^n-1}=0$,
		\item $u^{p^i-1}f_{i,j}=0$,
		\item $\alpha_If_I=0$ whenever $n\in I$,
		\item $f_I^2=0$ whenever $n\in I$,
		\item $f_{i,j}^2=0$ for $p = 2$,
		\item $f_{1,2,3}^2=0$.
	\end{enumerate}
\end{lemma}
Since we are working with graded commutative rings, squares of odd degree elements are always $0$, except when $p=2$.
The only other relation of the form $xy=0$ that is not covered by \Cref{prop: mult extensions} is $u^{p-1}f_{1,2,3}=0$ when $n=4$.
For larger $n$ there are other possible multiplicative extensions that we are unable to address.
\begin{proof}
	\begin{asparaenum}
	\item The relations regarding $u$ and $u_i$ hold because they take place in filtration $0$ and so there is no room for possible extensions.
	
	\item Recall from the proof of \Cref{thm: collapse} that $u^{p^i-1}f_{i,j}=0$ as it is in odd total degree and the only elements in lower filtration are in even total degree.
	This is also our base case for the induction proof of the next relation.
	
	\item In each of the following cases all we have to do is show that there are no eligible candidates in the given total degree of lower filtration.
	Since we have a basis for $\Tor$, and hence for $\pi_*(\HF_p \wedge_{\BPn} e_n)$, this amounts to ruling out classes of the form $qu^kf_I$ where $q\in \pi_0(e_n)$ and $k$ and $I$ are of the appropriate degree.
	Sometimes $q$ will not play a role as $u^kf_I$ is already $0$.
	
	First let us consider $\alpha_I f_I=0$ when $n\in I=\{i_1,i_2,i_3,\ldots, i_m\}$.
	Assume that this is true for all $J$ of cardinality less than $m$.
	Thus we are looking for an element in total degree $(-1)+\sum_{i\in I}2p^i-1$ in filtration less than $m-1$.
	This is impossible for degree reasons.
	Since the parity of the total degree is the same as the parity of $\#I-1$ we see that the only way to have an element in the same total degree is to be in an even number of filtrations lower.
	So the next element in a lower filtration of the highest possible total degree is $qu^jf_{I''}$ for some $j$, where $I''$ is $I$ without its two smallest elements.
	The difference in total degree between  $f_{I''}$ and $\alpha_I f_I$ is $2p^{i_1}-1+2p^{i_2}-1+2p^{i_3}-1$ thus $j>p^{i_3}-1$.
	However, this product is already $0$ by our induction hypothesis.
	This relies on the already observed fact that $u^{p^i-1}f_{i,j}=0$.
	
	\item The relation $f_I^2$ is more straightforward.
	Suppose that $qu^kf_J$ is in the same total degree as $f_I^2$, which is $(-2)+\sum_{i\in I}4p^i-2$.
	By the relation $\alpha_Jf_J=0$ we see that $k<p^{j_1}-1$.
	Therefore the total degree of $qu^kf_J$ is less than $\sum_{j\in J}(2p^j-1) -1$.
	However this will never be large enough as the total degree of $f_I^2$ is larger than $4p^n-2$ and we have
	\[\begin{aligned}
	4p^n-2 > \sum_{i=1}^n2p^i-1
	> \sum_{j\in J}2p^j-1
	>  \vert qu^k f_J\vert.
	\end{aligned}\]
	
	\item Assume that $p = 2$ and consider the class $f_{i,j}$ in total degree $2^{j+1}-1$.
	Its square is in filtration $2$ and total degree $2\cdot 2^{j+1}-2$ and so $qu^{2^{j+1}-1}$ is the only possible class other than $0$ that $f_{i,j}^2$ could be.
	If $j=n-1$ then this power of $u$ is $0$.
	If $j\neq n-1$ then we obtain the relation $u^{2^i-1}qu^{2^{j+1}-1}=0$ as we know that $u^{2^i-1}f_{i,j}=0$.
	But this cannot be $0$ unless $q$ is divisible by $u_k$ for $k<i$.
	If this were the case though then $qu^{2^{j+1}-1}=0$ since $i<j$.
	
	\item The last case is the square of the element $f_{1,2,3}$ in total degree $2p^2-1+2p^3-1$.
	This is covered by other cases except when $n=4$.
	The possible elements in the same total degree are $a:=qu^{2p^2-1+2p^3-1}$ and $qu^mf_{i,j,k}$ in total degree $2m+2p^j-1+2p^k-1$ for $m<p^i-1$.
	We will deal with these two cases separately.
	
	First let us consider the case $a$.
	Note that $q=1\in \W(\F_{p^n})$ since if it were divisible by $u_i$, then $a=0$ because $u_iu^{p^i-1}=0$.
	At the prime $2$ the element $a=0$.
	Otherwise we have that $u^{p-1}a=0$.
	This contradicts the fact that $u^{p^4-2}\neq 0$ when $p>2$.
	
	Now consider the possibility that $f_{1,2,3}^2=qu^mf_{i,j,k}$.
	This element is annihilated by $u^{p-1}$ since $f_{1,2,3}$ is.
	If $k\neq 4$ then $f_{i,j,k}$ must be $f_{1,2,3}$ and $m<p-1$ so the element $qu^mf_{i,j,k}$ will not be in high enough total degree.
	However, the element $f_{i,j,4}$ has higher total degree than $f_{1,2,3}^2$ unless $p=2$.
	When $p=2$, $\vert f_{1,2,3}^2\vert =44$ and the only element in this total degree is $qu^3f_{1,2,4}$ which is $0$ by the above relation $\alpha_I f_I=0$ when $n\in I$.
	\qedhere
	\end{asparaenum}
\end{proof}

Now we establish that $\HH_*(e_n;\F_p)$ splits as a tensor product.
The following ring maps will help us split the homology of connective Morava $\moravaE$-theory.
\[
\pi_*(\HF_p\wedge_{\BPn} e_n) \sr{\psi}\lla \HH_*(e_n;\F_p) \sr{\varphi}\lra \HH_*(\BPn;\F_{p^n})
\]
The ring map $\psi$ is induced by the maps from $S$ to $\BPn$,
where $S$ is the sphere spectrum, and therefore $\psi$ is a map of rings.
We will compute $\psi$ by considering the maps $BP \lra \BPn$ and the map from $\HF_p\wedge BP$ to $\HF_p$.
We have the map 
\[
\varphi': \HH_*(e_n;\F_p)\lra \HH_*(\HF_{p^n};\F_p)
\]
which is obtained by first taking $0$th Postnikov sections and then taking the quotient by the maximal ideal.
The map $\varphi$ is constructed by noting that the image of $\varphi'$ is contained in the image of $\HH_*(\BPn;\F_{p^n})$ after applying the twist map.
To compute each map involved we will consider the relevant map of spectral sequences where the $\Page{2}$-pages can always be computed using the ``same'' underlying Koszul complex.
Here we record some basic facts about the above maps.

\begin{lemma}\leavevmode
\label{lem: maps for splitting}
	\begin{enumerate}
	\item The map $\psi$ takes all generators, other than the unit, coming from $\HH_*(BP;\F_p)$ as well as the $\ov_I$ to $0$.
	\item The map $\varphi'$ takes each $u,u_i, f_I \in \HH_*(e_n;\F_p)$ to $0$.
	\item The classes $\ov_{n+k}\in \Tor^{BP_*}_1(\HH_*(BP;\F_p),e_{n*})$ are sent by $\varphi$ to the conjugates of the classes $\xi_{n+k+1}$ or $\tau_{n+k}$ in the dual Steenrod algebra, when the prime is $2$ or odd respectively.
	\item The map $\varphi'$ factors as
	\[
	\xymatrix{
	\HH_*(e_n;\F_p) \ar[rr]^{\varphi'} \ar[dr]_{\varphi}&
	&
	\HH_*(\HF_{p^n};\F_p)\\
	&
	\HH_*(\BPn;\F_{p^n}). \ar[ur]&
	}
	\]
	\end{enumerate}
\end{lemma}
These facts are proved by showing that the map of $\Page{2}$-pages takes the classes to $0$ and then realizing that there is nothing in lower filtration for the classes to be detected by on the $\Page{\infty}$-page.
\begin{proof}
	\begin{asparaenum}
	\item The first claim is obvious by considering the following map of spectral sequences
	\[
	\Tor^{BP_*}(\HH_*(BP;\F_p),e_{n*})\lra \Tor^{BP_*}(\F_p,e_{n*}) \lra \Tor^{\BPn_*}(\F_p,e_{n*})
	\]
	each of which comes from a map of ring spectra.
	All of the above spectral sequences collapse since the first does and the maps are surjective on $\Tor$.
	The first map of spectral sequences takes all generators, other than the unit, of $\HH_*(BP;\F_p)$ to zero.
	The second map establishes the second claim as $\ov_{n+k}$ isn't a cycle here and there are no elements in lower filtration for the $\ov_{n+k}$ to be sent to as the $\ov_{n+k}$ are on the $1$-line.
	The claim for $\ov_I$ follows as all such maps are multiplicative.
	
	\item Clearly $\varphi'(u)=0=\varphi'(u_i)$.
	The $f_{i,j}$ is represented by a cycle which is expressed in terms of multiples of $u_i$ and $u_j$.
	Thus the $f_{i,j}$'s must be sent to $0$ on the $\Page{2}$-page.
	There is nothing in lower filtration and the same internal degree of the spectral sequence and so $\varphi'(f_{i,j})=0$.
	Since the map $\varphi'$ is induced by a map of commutative ring spectra, it takes Toda brackets to Toda brackets.
	Thus $\varphi(f_I)\in\varphi(\<f_{i,j},\alpha_{I'},f_{I'}\>)\subset\<0,0,0\>=\{0\}$ where $I'=I\setminus\{i,j\}$.
	\item We establish the third claim by considering the map of spectral sequences induced by 
	\[
	e_n \lra \HF_{p^n}.
	\]
	We use the same Koszul complex to compute $\Tor$.
	Each class $f_I$ is taken to zero, as established above.
	We also understand the spectral sequence
	\[
	\Tor^{BP_*}_s(\HH_*(BP;\F_p),\F_{p^n}) \ssra \HH_*(\HF_p;\F_{p^n})
	\]
	completely as we know what it converges to and so it must collapse.
	That these classes detect the conjugates follows from the discussion in Chapter 4 Section 2 starting on page 114  of \cite{GreenBook}.
	
	\item We have that $\HH_*(e_n;\F_p)$ is multiplicatively generated by
	\begin{itemize}
	    \item the classes $u_i,u$ coming from $e_{n*}$,
	    \item the classes $f_I$,
	    \item the classes $\ov_{n+k}$,
	    \item classes coming from $\HH_*(BP;\F_p)$, which are detected on the $0$-line.
	\end{itemize}
	We have already established $\varphi'$ of the first three collections actually lift to $\HH_*(\BPn;\F_{p^n})$.
	The last item follows by considering the map of spectral sequences
	\[
	\Tor^{BP_*}_*(\HH_*(BP;\F_p),e_{n*})\lra \Tor^{BP_*}_*(\HH_*(BP;\F_p),\F_{p^n})
	\]
	restricted to the $0$-line.
	This map of rings is induced by the map of commutative $S$-algebras $e_n \lra \HF_{p^n}$ and so induces a map of rings on the $0$-line.
	\qedhere
\end{asparaenum}
\end{proof}

\begin{prop}\label{cor: mult splitting}
	When $n\leq 4$ the homology $\HH_*(e_n;\F_p)$ splits as a tensor product of rings
	\[
	\HH_*(e_n;\F_p)\iso \HH_*(\BPn;\F_{p^n})\tensor_{\F_{p^n}} \pi_*(\HF_p\wedge_{\BPn}e_n).
	\]
\end{prop}

\begin{proof}
	We have the two maps of rings $\varphi$ and $\psi$.
	They induce a map of rings
	\[
	\HH_*(e_n;\F_p) \sr{\varphi\times\psi}\lra \HH_*(\BPn;\F_{p^n}) \times \pi_*(\HF_p \wedge_{\BPn} e_n).
	\]
	This composed with the canonical map
	\[
	\HH_*(\BPn;\F_{p^n}) \times \pi_*(\HF_p \wedge_{\BPn} e_n) \lra \HH_*(\BPn;\F_{p^n}) \tensor_{\F_{p^n}} \pi_*(\HF_p \wedge_{\BPn} e_n)
	\]
	gives us the desired splitting.
	To see that it is an isomorphism we note that it is injective and surjective on the generators and it is a map of $\F_{p^n}$-algebras.
	The injectivity and surjectivity follows from the collapse of the spectral sequence.
\end{proof}
Note that this is a splitting of rings, as there are no known maps of spectra  
$\HF_p \wedge e_n \lra \HF_p \wedge \BPn$
or $\HF_p \wedge_{\BPn} e_n \lra \HF_p \wedge e_n$.
Thus this splitting does not respect higher multiplicative structure, such as power operations.
However, it does resolve various extension problems relating the two tensor factors.

\printbibliography

\end{document}